\documentclass[a4paper,openany,12pt]{article} % a4paper,openany,12pt cctart
\usepackage{amsmath,amsfonts,amssymb,amsthm}
\usepackage{amsmath,amssymb,fancyhdr, mathrsfs, graphicx, epsfig}
\usepackage{enumerate}
\newcommand{\bc}{\begin{center}}
\newcommand{\ec}{\end{center}}
\newcommand{\ba}{\begin{array}}
\newcommand{\ea}{\end{array}}

\renewcommand{\dfrac}{\displaystyle\frac }

\newcommand{\dint }{\displaystyle\int }

\newtheorem{thm}{Theorem }[section]
\newtheorem{rem}{Remark}[section]
\newtheorem{cor}{Corollary }[section]

\newtheorem{lem}{Lemma }[section]

\newtheorem{proposition}{ Proposition }[section]

\numberwithin{equation}{section}
\numberwithin{thm}{section}
\numberwithin{defn}{section}
\numberwithin{lem}{section}
\numberwithin{cor}{section}
\numberwithin{rem}{section}

\title{Existence of infinitely many  solutions for the fractional
Schr\"odinger- Maxwell equations
     \thanks{ E-mail address: jnwzl32@163.com (Z.L. Wei).} \thanks
{Research supported by  the  NSF of
Shandong  Province (ZR2013AM009). }
}

\author{Zhongli Wei$^{a,b}$ \ \  \\
\small  School of  Sciences,
Shandong Jianzhu University,\\
  \ \  \small Jinan,  Shandong, 250101,
    People's
 Republic  of  China
    \\ \small $^{b}$ School of Mathematics,
Shandong University, \\ \small Jinan, Shandong 250100,  People's
Republic of China.  .}

\date{}
\begin{document}
\maketitle

\baselineskip 18pt

\begin{abstract}
In this paper, by using variational methods and critical point
theory, we shall mainly  study  the existence of
infinitely many solutions for the following fractional Schr\"odinger-Maxwell
equations
$$
   ( -\Delta )^{\alpha} u+V(x)u+\phi u=f(x,u),  \hbox{in }   \mathbb{R}^3 ,
$$
$$
(-\triangle)^{\alpha}\phi =K_{\alpha}  u^2 \ \  \mathrm{in}\ \  \mathbb{R}^3
$$
where $\alpha \in (0,1],$  $K_{\alpha}=\dfrac{\pi^{-\alpha}\Gamma(\alpha)}{\pi^{-(3-2\alpha)/2}\Gamma((3-2\alpha)/2)},$ $( -\Delta )^{\alpha}$   stands for the fractional Laplacian.   Under some more
assumptions on $f,$ we get infinitely many solutions for the system.
\end{abstract}

\\

 %% PACS codes here, in the form: \PACS code \sep code

%% MSC codes here, in the form: \MSC code \sep code
%% or \MSC[2008] code \sep code (2000 is the default)

{\bf 2000 MR. Subject Classification} \ \ 35B40, 35B45, 35J55, 35J60, 47J30.

\section{Introduction and the Main Result}

In this paper, we study the  fractional Schr\"{o}dinger-Maxwell equations
\begin{equation}\label{FSMeq 1.1}
   ( -\Delta )^{\alpha}u+V(x)u+\phi u=f(x,u),  \hbox{in \ \ $ \mathbb{R}^3$ ,}
\end{equation}
 \begin{equation}\label{FSMeq 1.2}
(-\triangle)^{\alpha}\phi =K_{\alpha}  u^2 \ \  \mathrm{in}\ \  \mathbb{R}^3
\end{equation}
 where $ u,\phi:\mathbb{R}^3\rightarrow \mathbb{R},f: \mathbb{R}^3\times\rightarrow \mathbb{R}, $ $\alpha \in (0,1],$  \   $K_{\alpha}=\dfrac{\pi^{-\alpha}\Gamma(\alpha)}{\pi^{-(3-2\alpha)/2}\Gamma((3-2\alpha)/2)},$ $ (-\triangle)^{\alpha}$  stands for the fractional Laplacian.
 Here the
fractional Laplacian \  $(-\triangle)^{\alpha}$ with $\alpha \in (0,1] $ of a function $\phi:\mathbb{R}^3\rightarrow \mathbb{R}$  is defined by: $$\mathcal {F}((-\triangle)^{\alpha}\phi)(\xi)=|\xi|^{2\alpha}\mathcal {F}(\phi)(\xi),\ \forall \alpha\in
(0,1],$$
where  $\mathcal
{F}$ is the Fourier
transform, i.e.,\  $$\mathcal
{F}(\phi)(\xi)=\frac{1}{(2\pi)^{3/2}}\int_{\mathbb{R}^3}\exp\{-2\pi
i\xi\cdot x\}\phi(x)\mathrm{d} x.$$ If $\phi$   is smooth enough,  $(-\Delta )^\alpha$  can also be computed
by the following singular integral
:\ $$(-\triangle)^{\alpha}\phi(x)=c_{3,\alpha}\mathrm{P.V.}\int_{\mathbb{R}^3}\frac{\phi(x)-\phi(y)}{|x-y|^{3+2\alpha}}\mathrm{d}
y.$$
Here $\mathrm{P.V.}$ is the principal value and $c_{3,\alpha}$ is a normalization constant.
Such a system \eqref{FSMeq 1.1} is called  Schr\"{o}dinger-Maxwell equations or Schr\"{o}dinger-Poisson equations which is obtained while looking
for existence of standing waves for the fractional nonlinear Schr\"{o}dinger equations interacting with an unknown electrostatic field. For
a more physical background of system \eqref{FSMeq 1.1}, we refer the reader to \cite{DM PRSE 2004 134,BF TMNA 1998 11}  and the references therein.

When $\alpha =1$,   system \eqref{FSMeq 1.1}  was first introduced by Benci and Fortunato in \cite{DM PRSE 2004 134}, and it
has been widely studied by many authors; The case $V \equiv 1$ or being radially
symmetric, has been studied under various conditions on $f$ in \cite{AR CCM 2008 10}-\cite{LZ FZ NA 2009 70};
When $V(x)$  is not a constant, the existence of infinitely many large solutions for
\eqref{FSMeq 1.1}  has been considered in \cite{AP jmaa 2008 345}-\cite{HuangWN TangXH RM 2014 65} via the fountain theorem (cf.  \cite{WillemM 1996,ZouWM MM 2001 104}.)

In  system \eqref{FSMeq 1.1}, we assume the following hypotheses on potential $V$ and nonlinear term $f :$ \\[2mm]
 ($\mathbb{V}$)\   $V\in C(\mathbb{R}^3,\mathbb{R})$, $\inf_{x\in \mathbb{R}^3}V(x)\geq a_1>0,$
where $a_1$  is a positive constant. Moreover,   $\lim\limits_{|x|\rightarrow \infty } \ V(x)=+ \infty.$\\[2mm]
$(\mathbb{H}_1)$ \ \  $f\in C( \mathbb{R}^3\times \mathbb{R},\mathbb{R})$, and there exist  $c_1,c_2>0,\, p\in(4,2_\alpha^*)$  such that
$$|f(x,u)|\leq c_1|u|+c_2| u|^{p-1}, \ \
 \forall \ x\in \mathbb{R}^3,\ u\in\mathbb{R}, $$
where, $2_\alpha^*=\frac{6}{3-2 \alpha}, \ \alpha> \frac{3}{4},$  $f(x,u)u\geq 0$ for $u\geq 0$.\\
$(\mathbb{H}_2)$ \ \ $\lim_{ |u|\to\infty}\frac{F(x,u)}{ u^4}=+\infty\ $  uniformly for $ x\in \mathbb{R}^3,$ here
$F(x,u)=\int_0^u f(x,t)\, \mathrm{d}t.$
\\
$(\mathbb{H}_3)$ \ \
Let  $G(x,u)=\frac 1 4f(x,u)u-F(x,u),$  there exist $a_0>0, $ and $g(x)\geq 0$ such that
$\int_{\mathbb{R}^3}g(x) \mathrm{d}x <+ \infty, $
$G(x,u)\geq -a_0g(x),\quad\forall\ (x,u)\in \mathbb{ R}^3\times\ \mathbb{R}.$
\\
$(\mathbb{H}_4)$ \ \
  $f(x,-u)=-f(x,u)\ \forall\ x\in \mathbb{R}^3,\ u\in \mathbb{R} .$

Now, we are ready to state the main result of this paper.
\begin{thm}\label{FSMeq thm1.1} Assume that $(\mathbb{V})$ and $(\mathbb{H}_1)-(\mathbb{H}_4)$  satisfy. Then system \eqref{FSMeq 1.1}
possesses infinitely many nontrivial solutions.
\end{thm}
\begin{rem} $(i):$ \
There are functions $f$ satisfying the assumptions $(\mathbb{H}_1)-(\mathbb{H}_4),$   for example $(1): \ f(x,u)=4u^3\ln(u^2+1)+\frac{2u^5}{(u^2+1)},$ then $a_0=0,$ $(\mathbb{H}_3)$  is satisfied;   $(2): \ f(x,u)=e^{-\sum_{i=1}^3|x_i|}u+|u|^{p-2}u,\ p\in(4,2_\alpha^*),\ \alpha >\frac{3}{4},$ then  $a_0=\frac{r_0^2}{4}, g(x)=e^{-\sum_{i=1}^3|x_i|}, r_0=\left(\frac{p}{p-4}\right)^{1/(p-2)}+1,$ $(\mathbb{H}_3)$  is satisfied.\\[2mm]
$(ii):$ \ the assumption $(\mathbb{H}_3)$  is weaker than the assumptions  $(f_4)$ in paper $\cite{LiQD Weizl NA 2010 72}$ and $(f3')$ in paper $\cite{HuangWN TangXH RM 2014 65}.$
\end{rem}
\section{Variational settings and preliminary results}

Now, let$'$s introduce some notations. For any $1 \leq  r < \infty, L^r(\mathbb{R}^3)$  is the
usual Lebesgue space with the norm
$$\|u\|_{L^r}=\Big(\int_{ \mathbb{R}^3}|u(x)|^r\,\mathrm{d}x\Big)^{\frac 1 r}.$$
The fractional order Sobolev space:
$$ H^\alpha (\mathbb R^3)=\left\{u\in L^2(\mathbb R^3):\ \int_{\mathbb{R}^3}(|\xi|^{2\alpha}\hat{u}^2+\hat{u}^2)\ \mathrm{d} \xi<\infty\right\},$$
where \ $\hat{u}=\mathcal {F}(u)$,  The norm is defined by
$$ \|u\|_{H^\alpha(\mathbb
R^3)}=\left(\int_{\mathbb{R}^3}(|\xi|^{2\alpha}\hat{u}^2+\hat{u}^2)\ \mathrm{d}
\xi\right)^{\frac{1}{2}}.$$
The spaces
$D^\alpha (\mathbb R^3)$  is defined as the completion of $C_0^{\infty}(\mathbb R^3)$ under the norms
$$ \|u\|_{D^\alpha(\mathbb
R^3)}=\left(\int_{\mathbb{R}^3}(|\xi|^{2\alpha}\hat{u}^2 \mathrm{d}
\xi\right)^{\frac{1}{2}}=\left(\int_{\mathbb{R}^3}|(-\Delta)^{\alpha/2}u(x)|^2 \mathrm{d}
x\right)^{\frac{1}{2}}.$$
Note that, by Plancherel's theorem we have  $\|u\|_{2}=\|\hat{u}\|_2,$  and
 $$\aligned \int_{\mathbb{R}^3}|(-\Delta)^{\frac{\alpha}{2}}u(x)|^2 \mathrm{d}  x&=\int_{\mathbb{R}^3}
(\widehat{(-\Delta)^{\frac{\alpha}{2}}u(\xi)})^2  \mathrm{d} \xi
 =\int_{\mathbb{R}^3}(|\xi|^\alpha \hat{u}(\xi))^2 \mathrm{d}
\xi \\ & =\int_{\mathbb{R}^3}|\xi|^{2\alpha}\hat{u}^2 \mathrm{d}  \xi<\infty, \ \forall
u\in H^\alpha(\mathbb R^3).\endaligned $$
It follows that
$$ \|u\|_{H^\alpha(\mathbb
R^3)}=\left(\int_{\mathbb{R}^3}(|(-\Delta)^{\frac{\alpha}{2}}u(x)|^2 +u^2)\mathrm{d}  x \right)^{\frac{1}{2}}.$$
In our problem, we work in the space defined by
\begin{equation}\label{FSMeq-space}
E:=\bigg\{u\in H^\alpha(\mathbb R^3)\Bigm|\left(\int_{\mathbb{R}^3}(|(-\Delta)^{\frac{\alpha}{2}}u(x)|^2 +V(x)u^2)\mathrm{d}  x \right)^{\frac{1}{2}}<\infty\bigg\}.
\end{equation}
Thus, $E$ is a Hilbert space with the inner product
$$(u,v)_E:=\int_{ \mathbb{R}^3}\big((-\Delta)^{\frac{\alpha}{2}}u(x)\cdot(-\Delta)^{\frac{\alpha}{2}}v(x) +V(x)uv\big)\,\mathrm{d}x.$$
and its norm is $\|u\|=(u,u)^{\frac 12}.$  Obviously, under the assumptions $(\mathbb{V}),$
$\|u\|_E \equiv \|u\|_{H^{\alpha}}. $

 \begin{lem}[see  $\cite{ChangXJ JMP2013}$\ Lemma  2.2  and   $\cite{PFelmer AQuaas and JGTan PRSE 2012 142}$]\label{FLKGSlem2.1}
     $H^\alpha(\mathbb{R}^3)$ is continuously embedded into $L^p(\mathbb{R}^3)$  for $p\in [2,2^*_{\alpha}];$  and compactly embedded into
$L^p
_{loc}(\mathbb{R}^N)$ for  $p\in [2,2^*_{\alpha})$ where $2^*_{\alpha}=\dfrac{6}{3-2\alpha}.$
 Therefore, there exists a positive constant $C_p$ such that
$$\|u\|_p\leq C_p \|u\|_{H^\alpha(\mathbb{R}^3)}. $$
 \end{lem}
\begin{lem}[see $\cite{Zifei Shen and Fashun Gao AAA 2014}$]\label{FLKGSlem2.2}
Under the assumption $(\mathbb{V}),$ the embedding  $E$   is compactly embedded into  $L^p(\mathbb{R}^3)$  for $p\in [2,2^*_{\alpha}).$
  \end{lem}
 \begin{lem}[see $\cite{HajaiejH YuX ZhaiZ jmaa 2012}$]\label{FLKGSlem2.3}
 For $1<p<\infty$ and $0<\alpha < N/p,$  we have
 \begin{equation}\label{FLKGS 2.1}
    \|u\|_{L^{\frac{pN}{N-p\alpha}}(\mathbb{R}^N)}\leq B \|(-\Delta)^{\alpha/2}u\|_{L^p(\mathbb{R}^N)}
 \end{equation}
 with best constant
 $$
 B=2^{-\alpha}\pi^{-\alpha/2}\frac{\Gamma((N-\alpha)/2)}{\Gamma((N+\alpha)/2)}
 \left(\frac{\Gamma((N)}{\Gamma(N/2)}\right)^{\alpha/N}.
 $$
  \end{lem}
 \begin{lem}\label{FLKGSlem2.4}
  For any  $u\in H^\alpha(\mathbb{R}^N)$ and for any $h\in D^{-\alpha}(\mathbb{R}^N),$  there exists a unique solution $\phi=\left((-\Delta)^\alpha+u^2\right)^{-1}h\in D^\alpha(\mathbb{R}^N)$  of the equation
  $$
  (-\Delta)^\alpha \phi +u^2 \phi =h,
  $$
(being $ D^{-\alpha}(\mathbb{R}^N)$  the dual space of $ D^{\alpha}(\mathbb{R}^N)$).
Moreover, for every $u\in H^\alpha(\mathbb{R}^N)$ and for every
$h,g \in D^{-\alpha}(\mathbb{R}^N),$
\begin{equation}\label{FLKGS 2.2}
   \langle h, \left((-\Delta)^\alpha+u^2\right)^{-1}g \rangle
   =\langle g, \left((-\Delta)^\alpha+u^2\right)^{-1}h \rangle
 \end{equation}
where $\langle\cdot, \cdot\rangle$  denotes the duality pairing between $ D^{-\alpha}(\mathbb{R}^N)$ and $ D^{\alpha}(\mathbb{R}^N).$
  \end{lem}
\begin{proof}[\bf Proof.]  If $u\in H^\alpha(\mathbb{R}^N),$  then by H\"{o}lder inequality and $\eqref{FLKGS 2.1}$
\begin{equation}\label{FSMeq 2.4}
 \int_{\mathbb{R}^N} u^2\phi^2\mathrm{d}x \leq \|u\|^2_{2p}\|\phi\|^2_{2q}\leq B^2  \|u\|^2_{2p}\|\phi\|^2_{D^\alpha},
\end{equation}
where $\frac{1}{p}+\frac{1}{q}=1, \ q=\frac{N}{N-2\alpha},\ 2q=2^*_{\alpha}.$
 Thus
$\left(\int |(-\Delta)^{\alpha/2}\phi|^2+\int u^2\phi^2\right)^{1/2}$
is a norm in  $ D^{\alpha}(\mathbb{R}^N)$  equivalent to  $ \|\phi \|_{D^{\alpha}}.$
Hence, by the application of Lax-Milgram Lemma, we obtain  the existence part.
For every $u\in H^\alpha(\mathbb{R}^N)$ and for every
$h,g \in D^{-\alpha}(\mathbb{R}^N),$ we have
$\phi_g= \left((-\Delta)^\alpha+u^2\right)^{-1}g, \ \phi_h=\left((-\Delta)^\alpha+u^2\right)^{-1}h.$
Hence,
$$ \ba{l}\langle h, \left((-\Delta)^\alpha+u^2\right)^{-1}g \rangle
=\dint h\left((-\Delta)^\alpha+u^2\right)^{-1}g \mathrm{d}x \\[3mm]
=\dint h\phi _g \mathrm{d}x =\int \left((-\Delta)^\alpha+u^2\right)\phi _h\phi _g \mathrm{d}x \\[3mm]
 =\dint \left((-\Delta)^\alpha \phi _h+u^2\phi _h\right)\phi _g \mathrm{d}x
 =\dint \left((-\Delta)^\alpha \phi _g+u^2\phi _g\right)\phi _h \mathrm{d}x
 \\[3mm]= \dint g \phi _h \mathrm{d}x = \int g \left((-\Delta)^\alpha+u^2\right)^{-1}h \mathrm{d}x
   =\langle g, \left((-\Delta)^\alpha+u^2\right)^{-1}h \rangle. \ea
$$
So, we get \eqref{FLKGS 2.2}.\end{proof}

\begin{lem}[see $\cite{Elliott H Lieb  Michael Loss Analysis 2001}$]\label{FLKGS lem4.1}
Let $f$ be a function in $C_0^{\infty}(\mathbb{R}^N)$  and let $0 < \alpha < n .$  Then, with
 \begin{equation}\label{FLKGS 4.1}
    c_{\alpha}\doteq \pi ^{-\alpha/2}\Gamma(-\alpha/2),
 \end{equation}
 \begin{equation}\label{FLKGS 4.2}
    c_{\alpha}(\xi ^{-\alpha}\widehat{f}(\xi))^\vee{}(x)=  c_{n-\alpha} \int_{\mathbb{R}^n}|x-y|^{\alpha -n}f(y)\mathrm{d}y.
 \end{equation}
   \end{lem}
  \begin{lem}\label{FLKGS lem4.2}
  For every $u\in H^\alpha$  there exists a unique $\phi =\phi(u) \in D^\alpha$  which solves
 equation  \eqref{FSMeq 1.2}. Furthermore,  $\phi(u)$ is given by
 \begin{equation}\label{FLKGS 4.3}
   \phi(u)(x)= \int_{\mathbb{R}^3}|x-y|^{2\alpha -3}u^2(y)\mathrm{d}y.
 \end{equation}
As a consequence, the map  $\Phi : \ u\in H^\alpha \mapsto \phi(u) \in D^\alpha $  is of class $C^1$ and
 \begin{equation}\label{FLKGS 4.4}
    [ \Phi(u)]'(v)(x)=2 \int_{\mathbb{R}^3}|x-y|^{2\alpha -3}u(y)v(y)\mathrm{d}y, \ \ \forall u,v \in H^\alpha.
 \end{equation}
   \end{lem}
\begin{proof}[\bf Proof.] The existence and uniqueness part follows by Lemma \ref{FLKGSlem2.4}. By Lemma \ref{FLKGS lem4.1} and the Fourier transform of  equation   \eqref{FSMeq 1.2}, the representation formula \eqref{FLKGS 4.3} holds for  $u\in C_0^{\infty}(\mathbb{R}^3);$  by density it can be extended for any  $u\in H^\alpha.$  The representation formula \eqref{FLKGS 4.4} is obvious.\end{proof}

  System \eqref{FSMeq 1.1} and  \eqref{FSMeq 1.2} are
the Euler-Lagrange equations corresponding to the functional  $J: H^\alpha(\mathbb
R^3)\times  D^\alpha(\mathbb
R^3)\rightarrow \mathbb{R}$ is
$$
J(u,\phi)=\frac{1}{2}\int_{\mathbb{R}^3}\left(|(-\Delta)^{\frac{\alpha}{2}}u(x)|^2 +V(x) u^2 -\frac{1}{2}|(-\Delta)^{\frac{\alpha}{2}}\phi(x)|^2+ K_{\alpha} \phi u^2\right)\mathrm{d}x -\int_{\mathbb{R}^3}
F(x,u)\mathrm{d}x,$$
 where $F(x,t)=\dint _0^t f(x,s)\mathrm{d}s, \  t\in \mathbb{R}.$\\[2mm]
 Evidently, the action functional $J$  belongs to $C^1(H^{\alpha}(\mathbb{R}^3)\times D^{\alpha}(\mathbb{R}^3),\mathbb{R})$ and
the partial derivatives in $(u, \phi)$ are given, for $\xi \in  H^{\alpha}(\mathbb{R}^3)$   and $\eta \in D^{\alpha}(\mathbb{R}^3),$ by
$$
\ba{l}
\left\langle\dfrac{\partial J}{\partial u}(u,\phi), \xi\right\rangle
 = \dint_{\mathbb{R}^3}\left((-\Delta)^{\frac{\alpha}{2}}u(x)(-\Delta)^{\frac{\alpha}{2}}\xi (x)
 + V(x)u\xi + K_{\alpha}  \phi u\xi \right)\mathrm{d}x -\int_{\mathbb{R}^3}
f(x,u)\xi(x) \mathrm{d}x, \\[4mm]
\left\langle\dfrac{\partial J}{\partial \phi}(u,\phi), \eta \right\rangle
 = \dfrac{1}{2}\dint_{\mathbb{R}^3}\left(-(-\Delta)^{\frac{\alpha}{2}}\phi(x)(-\Delta)^{\frac{\alpha}{2}}\eta (x)
+K_{\alpha}  u^2\eta \right)\mathrm{d}x.
\ea
$$
Thus, we have the following result:
 \begin{proposition}\label{propst2.2}
  The pair  $(u,\phi)$  is a weak solution of system  \eqref{FSMeq 1.1} and  \eqref{FSMeq 1.2} if and only if
it is a critical point of $J$ in  $ H^{\alpha}(\mathbb{R}^3)\times D^{\alpha}(\mathbb{R}^3).$
\end{proposition}
 So, we can consider the functional $J : H^{\alpha}(\mathbb{R}^3) \rightarrow \mathbb{R} $ defined by
 $J(u)=
J (u, \phi(u)). $  After multiplying  \eqref{FSMeq 1.2} by $\phi(u)$ and integration by
parts, we obtain
$$
\dint_{\mathbb{R}^3}|(-\triangle)^{\alpha/2}\phi(u)|^2 \mathrm{d}x= K_{\alpha}
\dint_{\mathbb{R}^3} \phi(u) u^2\mathrm{d}x
.
$$
Therefore, the reduced functional takes the form

 \begin{equation}\label{FLKGS 4.5}
  J(u)= \frac{1}{2}  \int_{\mathbb{R}^3}(|(-\Delta)^{\frac{\alpha }{2}}u(x)|^2
+V(x)u^2) \mathrm{d}  x
+\frac{1}{4} K_{\alpha}  \int_{\mathbb{R}^3}u^2 \phi(u)\mathrm{d}  x
-   \int_{\mathbb{R}^3}F(x,u)\mathrm{d}  x   .
 \end{equation}
 \begin{lem}\label{FLKGS lem4.3}
  Assume that there exist $c_1, c_2 > 0$ and $p > 1$ such that
 \begin{equation}\label{FLKGS 4.6}
  |f(s)| = c_1|s| + c_2|s|^{ p-1}, \ \ \forall  s \in  \mathbb{R}.
 \end{equation}
Then the following statements are equivalent:\\[2mm]
i) $(u, \phi)\in  (H^\alpha  \cap L^p) \times D^\alpha $ is a solution of the system $\eqref{FSMeq 1.1}-\eqref{FSMeq 1.2}; $\\[2mm]
ii)  $u\in  H^\alpha  \cap L^p  $  is a critical point of $J$ and $\phi =\phi(u).$
     \end{lem}
\begin{proof}[\bf Proof.]   By the assumption \eqref{FLKGS 4.6}, the Nemitsky operator $u\in  H^\alpha  \cap L^p  \mapsto  F(x,u) \in  L^1$
is of class $C^1.$ Hence, by Lemma \ref{FLKGS lem4.2}, for every $u, v\in  H^\alpha $
$$
\ba{lll}J'(u)[v]& =&  \dint_{\mathbb{R}^3}(-\Delta)^{\frac{\alpha }{2}}u(x)(-\Delta)^{\frac{\alpha }{2}}v(x) \mathrm{d}  x
 + \dint_{\mathbb{R}^3}V(x)uv \mathrm{d}  x\\[3mm]
 & & +\dfrac{1}{2} K_{\alpha}  \int_{\mathbb{R}^3}uv  \int_{\mathbb{R}^3}|x-y|^{2\alpha -3}u^2(y)\mathrm{d}y \mathrm{d}  x\\[3mm]
& & +\dfrac{1}{2} K_{\alpha}   \int_{\mathbb{R}^3}u^2  \int_{\mathbb{R}^3}|x-y|^{2\alpha -3}u(y)v(y)\mathrm{d}y  \mathrm{d}  x
-   \dint_{\mathbb{R}^3}f(x,u)v\mathrm{d}  x \\[3mm]
&=&    \dint_{\mathbb{R}^3}(-\Delta)^{\frac{\alpha }{2}}u(x)(-\Delta)^{\frac{\alpha }{2}}v(x) \mathrm{d}  x
 +  \int_{\mathbb{R}^3}V(x)uv \mathrm{d}  x\\[3mm]
& & + K_{\alpha} \dint_{\mathbb{R}^3}uv   \phi(u) \mathrm{d}  x
-   \int_{\mathbb{R}^3}f(x,u)v\mathrm{d}  x .\ea
$$
By Fubini-Tonelli's  Theorem,  we can obtain  the conclusion.\end{proof}

If $1 \leq  p < \infty$ and $a, b \geq  0,$ then
\begin{equation}\label{FSMeq 2.10}
 (a+b)^p\leq 2^{p-1}(a^p+b^p).
 \end{equation}
From \eqref{FSMeq 1.2} and  \eqref{FLKGS 2.1}, for any $u \in  E$ using H\"older inequality we
have
$$
\|\phi(u)\|^2_{D^\alpha}=K_{\alpha}\int_{\mathbb{R}^3}\phi(u)u^2\mathrm{d}  x\leq K_{\alpha}
\|\phi(u)\|_{q}\|u\|^2_{2p}\leq C\|\phi(u)\|_{D^\alpha}\|u\|^2_{2p}.
$$
 where $\frac{1}{p}+\frac{1}{q}=1, \ q= 2^*_{\alpha}=\frac{6}{3-2\alpha}, \ \alpha>\frac{3}{4}.$
 Here and subsequently, $C$ denotes an universal positive constant. This and lemma \ref{FLKGSlem2.2} implies
that
\begin{equation}\label{FSMeq 2.11}
 \|\phi(u)\|_{D^\alpha}\leq C\|u\|^2_{2p}\leq C\|u\|^2_{E},
 \end{equation}
 \begin{equation}\label{FSMeq 2.12}
 \int_{\mathbb{R}^3}\phi(u)u^2\mathrm{d}  x\leq C\|u\|^4_{2p}\leq C\|u\|^4_{E}.
 \end{equation}
\begin{lem}\label{FSM lem2.8}
 Assume that a sequence $\{u_n\} \subset  E,$ $u_n\rightharpoonup  u$ in $E $ as $n \rightarrow \infty $ and
$\{u_n\} $  be a bounded sequence. Then
$$
\left|\int_{\mathbb{R}^3}(\phi(u_n)u_n-\phi(u)u)(u_n-u)\mathrm{d}x\right| \rightarrow 0, \ \ \mathrm{as} \ \ n\rightarrow \infty.
$$
\end{lem}
\begin{proof}[\bf Proof.]   Let $\{u_n\} $  be a sequence satisfying the assumptions $u_n\rightharpoonup  u$ in $E $ as $n \rightarrow \infty $ and
$\{u_n\} $  is bounded. Lemma \ref{FLKGSlem2.2}  implies that $u_n \rightarrow u$ in $ L^r(\mathbb{R}^3),$
where $2 \leq  r < 2^*_{\alpha},$ and $u_n \rightarrow u$ for a.e. $x\in \mathbb{R}^3.$  Hence $\sup_{n\in \mathbb{N}}\|u_n\|_r<\infty$  and
$\|u\|_r$ is finite. By H\"{o}lder inequality, \eqref{FSMeq 2.10}, \eqref{FSMeq 2.11} and \eqref{FSMeq 2.4}
\begin{equation}\label{FSMeq 2.13}
\ba{l} \left|\dint_{\mathbb{R}^3}(\phi(u_n)u_n-\phi(u)u)(u_n-u)\mathrm{d}x\right|\\[4mm]
 \leq \left(\dint_{\mathbb{R}^3}(\phi(u_n)u_n-\phi(u)u)^2\mathrm{d}x\right)^{\frac{1}{2}}
 \left(\dint_{\mathbb{R}^3}(u_n-u)^2\mathrm{d}x\right)^{\frac{1}{2}}\\[4mm]
 \leq \left(2\dint_{\mathbb{R}^3}(|\phi(u_n)u_n|^2+|\phi(u)u|^2)\mathrm{d}x\right)^{\frac{1}{2}}
 \|u_n-u\|_2 \\[4mm] \leq C (\|u_n\|_E^6+\|u\|_E^6)^{\frac{1}{2}}\|u_n-u\|_2
  \rightarrow 0, \ \ \mathrm{as} \ \ n\rightarrow \infty.\ea
 \end{equation}
\end{proof}

\section{Proof of Theorem \ref{FSMeq thm1.1}}
We say that $J\in  C^1(X, \mathbb{R})$ satisfies the $(C)_c$-condition if any sequence $\{u_n\}$
such that
$$J(u_n) \rightarrow  c, \ \ \  \|J'(u_n)\| (1 + \|u_n\|) \rightarrow 0
$$
has a convergent subsequence, where $X$ is a Banach space.
\begin{lem}\label{FSMeq lem3.1} Assume that $(\mathbb{V})$ and $(\mathbb{H}_1)-(\mathbb{H}_4)$  satisfy.
Then any sequence
$\{u_n\}\subset E $ satisfying
$$J(u_n) \rightarrow c > 0, \ \ \ \langle J'(u_n),u_n\rangle \rightarrow
  0,$$
is bounded in $E.$ Moreover, $\{u_n\} $ contains a converge subsequence.
\end{lem}
\begin{proof}[\bf Proof.]
 To prove the boundedness of $\{u_n\} $, arguing by contradiction, suppose
that $\|u_n\|\rightarrow \infty $ as $n \rightarrow \infty .$  By $(\mathbb{H}_3)$ \
 for sufficiently large $n \in  \mathbb{N} $
 $$
 \ba{ll}
 c+1&\geq J(u_n)-\dfrac{1}{4}\langle J'(u_n),u_n\rangle\\[3mm]
  & = \dfrac{1}{4} \|u_n\|^2 +\dint_{\mathbb{R}^3} G(x, u_n)\mathrm{d}x \\[3mm]
  & \geq  \dfrac{1}{4} \|u_n\|^2  -a_0\int_{\mathbb{R}^3}g(x) \mathrm{d}x\rightarrow +\infty.
 \ea
  $$
  Thus $\sup_{n \in  \mathbb{N} }\|u_n\| < \infty.$ i.e.  $\{u_n\} $ is a bounded sequence.

Now we shall prove  $\{u_n\} $ contains a subsequence, without loss of generality, by Eberlein-Shmulyan theorem (see for instance in \cite{Yosida K 1999}), passing to a
subsequence if necessary, there exists a $u \in E$ such that $ u_n \rightharpoonup  u \ \mathrm{in}\  E, $ again
by Lemma \ref{FLKGSlem2.2},  $ u_n \rightarrow  u \ \mathrm{in}\  L^s(\mathbb{R}^3), $   for $2 \leq  s < 2^*_{\alpha}$  and  $ u_n \rightarrow  u  $  a.e.  $x\in   \mathbb{R}^3. $  By $(\mathbb{H}_1)$ and using H\"older inequality we
have
 $$
 \ba{l}
 \left|\dint_{\mathbb{R}^3} (f(x,u_n)-f(x,u))(u_n-u)\mathrm{d}x\right|\\[3mm]
\leq \dint_{\mathbb{R}^3}\left |c_1(|u_n|+|u|)+c_2(|u_n|^{p-1}+|u|^{p-1})\right||u_n-u|\mathrm{d}x\\[3mm]
  \leq c_1(\|u_n\|_2+\|u\|_2)\|u_n-u\|_2+c_2(\|u_n\|_p^{p-1}+\|u\|_p^{p-1})\|u_n-u\|_p\\[3mm]
\rightarrow 0,\  \mathrm{as} \ n \rightarrow \infty.
 \ea
  $$
  Since $J\in  C^1(E),$  we have $J'(u_n) \rightarrow J'(u)$ in  $E^*.$  i.e.
 $$ \langle J'(u_n)-J'(u),u_n-u\rangle \rightarrow 0,\  \mathrm{as} \ n \rightarrow \infty.
  $$
  This together with Lemma \ref{FSM lem2.8} implies
   $$
 \ba{ll}
 \|u_n-u\|^2&= \langle J'(u_n)-J'(u),u_n-u\rangle  -K_{\alpha}\dint_{\mathbb{R}^3}(\phi(u_n)u_n-\phi(u)u)(u_n-u)\mathrm{d}x\\[3mm]
 & + \dint_{\mathbb{R}^3} (f(x,u_n)-f(x,u))(u_n-u)\mathrm{d}x
 \rightarrow 0,\  \mathrm{as} \ n \rightarrow \infty.
 \ea
  $$
  That is   $ u_n \rightarrow  u $ in $E.$
\end{proof}
\begin{lem}\label{FSMeq lem3.1}
 Suppose that  assumptions $(\mathbb{V}),$ $(\mathbb{H}_1)$ and $(\mathbb{H}_2)$ satisfy,   for any finite dimensional
subspace $\widetilde{E}\subset E,$  there holds
\begin{equation}\label{FSMeq 3.2}
    J(u)\rightarrow - \infty, \ \ \ \|u\|\rightarrow  \infty, \ \ u\in \widetilde{E}.
\end{equation}
\end{lem}
\begin{proof}[\bf Proof.] Arguing indirectly, assume that for some sequence $\{u_n\}\subset
 \widetilde{E}$  with
 $\|u_n\|\to \infty,$ there is $M>0$  such that $J(u_n)\ge -M,$ $ \forall$ $n\in \mathbb N.$  Set
 $v_n=\frac{u_n}{\|u_n\|},$  then $\|v_n\|=1$.
 Passing to a subsequence, we may assume that $v_n\rightharpoonup v$
 in $E.$ Since $\dim  \widetilde{E}<\infty$, then $v_n\to v\in \widetilde{E} $, $v_n(x)\to v(x)$
a.e. on $x\in\mathbb R^3$, and so  $\|v\|=1$.  Let $\Omega:=\{x\in \mathbb R^3:v(x)\not
 =0\}$, then $\text{meas}(\Omega)>0$
and for a.e. $x\in \Omega$, we have $\lim_{n\to \infty}|u_n(x)|\to
 \infty.$  It
follows from \eqref{FLKGS 4.5}, \eqref{FSMeq 2.12}  that
 \begin{equation}\label{FSMeq 3.2}
 \lim_{n\to \infty}\dfrac{4\dint_{\mathbb{R}^3}
 F(x,u_n)\mathrm{d} x}{\|u_n\|^4}=\lim_{n\to \infty}\frac{2\|u_n\|^2+K_{\alpha} \dint_{\mathbb{R}^3}
 \phi(u_n)u_n^2\mathrm{d} x-4J(u_n)}{\|u_n\|^4}\leq C.
 \end{equation}
But by the non-negative of $F,$ ($(\mathbb{H}_2)$ and Fadou’s Lemma, for large $n$  we have

 $$
 \aligned
&\lim_{n\to \infty} \frac{4\dint_{\mathbb{R}^3}
 F(x,u_n)\mathrm{d}x}{\|u_n\|^4}\geq \lim_{n\to \infty}\int_\Omega
 \frac{4F(x,u_n)v_n^4}{u_n^4}\mathrm{d} x\\
 &\geq \liminf_{n\to \infty}\int_\Omega
 \frac{4F(x,u_n)v_n^4}{u_n^4}\mathrm{d} x\\
 & \ge \int_\Omega \liminf_{n\to \infty}
 \frac{F(x,u_n)v_n^4}{u_n^4}\mathrm{d} x\\
 &=\int_\Omega \liminf_{n\to \infty}
 \frac{F(x,u_n)}{u_n^4}[\chi_\Omega(x)]v_n^4\mathrm{d} x\to \infty, \ n\to \infty.
 \endaligned
 $$ This contradicts to \eqref{FSMeq 3.2}.
\end{proof}

\begin{cor}\label{FSMeq cor3.1}
 Under assumptions $(\mathbb{V}),$ $(\mathbb{H}_1)$ and $(\mathbb{H}_2),$  for any finite dimensional
subspace $\widetilde{E}\subset E,$ there is $ R = R(\widetilde{E} ) > 0$ such that
\begin{equation}\label{FSMeq 3.3}
    J(u)\leq 0, \ \  \ \forall  u\in \widetilde{E}, \ \|u\|\geq R .
\end{equation}
\end{cor}
Let $\{e_j\} $ is an orthonomormal basis of $E$ and define $ X_j = \mathbb{R}e_j, $
\begin{equation}\label{FSMeq 3.4}
    Y_k=\oplus_{j=1}^kX_j, \ \  \  Z_k=\oplus_{j=k+1}^\infty X_j, \ \ k\in \mathbb{N}.
\end{equation}

\begin{lem}\label{FSMeq lem3.3}
 Under assumptions $(\mathbb{V}),$   for  $2\leq r <2^*_{\alpha},$  we have
\begin{equation}\label{FSMeq 3.5}
   \beta_k(r)=\sup _{u\in Z_k,\|u\|=1} \|u\|_r\rightarrow 0, \ \ \ k\rightarrow  \infty.
\end{equation}
\end{lem}
\begin{proof}[\bf Proof.]
Since the embedding from $E$ into $L^r(\mathbb{R}^3)$ is compact, then Lemma \ref{FSMeq lem3.3}
can be proved by a similar way as Lemma 3.8 in \cite{WillemM 1996}.
\end{proof}

By Lemma \ref{FSMeq lem3.3}, we can choose an integer $m \geq  1$ such that
\begin{equation}\label{FSMeq 3.6}
 \|u\|_2^2\leq \frac{1}{2c_1} \|u\|^2, \ \ \  \|u\|_p^p\leq \frac{p}{4c_2} \|u\|^p, \ \ \ \forall u\in Z_m.
\end{equation}
\begin{lem}\label{FSMeq lem3.4}
 Suppose that assumptions $(\mathbb{V})$ and $(\mathbb{H}_1)$ are satisfied,  there exist constants $\rho,\delta > 0$
such that $J|_{\partial
B_\rho\cap Z_m}\geq \delta >0.$
\end{lem}
\begin{proof}[\bf Proof.] By $(\mathbb{H}_1),$ we have
\[F(x,u)\leq
\frac{c_1}{2}u^2+\frac{c_2}{p}|u|^{p}, \ \forall (x,u)\in  \mathbb{R}^3\times \mathbb R.
\]
Hence, by
\eqref{FLKGS 4.5} and \eqref{FSMeq 3.6}, we have
$$\aligned J(u)& =\frac{1}{2}\|u\|^2+\frac{1}{4}K_{\alpha} \int_{\mathbb R^3}\phi(u)u^2\mathrm{d}x -\int_{\mathbb R^3} F(x,u)\mathrm{d} x \\
& \geq \frac{1}{2}\|u\|^2 -\int_{\mathbb R^3} F(x,u)\mathrm{d} x
\\
& \geq \frac{1}{2}\|u\|^2-\frac{c_1}{2}\|u\|_2^2-\frac{c_2}{p}|\|u\|_p^p\\
& \geq
\frac{1}{4}(\|u\|^2-\|u\|^p).
\endaligned
$$
Hence for any given $0 < \rho < 1 ,$  let $\delta =\frac{1}{4}(\rho^2-\rho ^p),$  then Φ$J|_{\partial
B_\rho\cap Z_m}\geq \delta >0.$
This complete the proof.
\end{proof}
\begin{lem}[see\cite{BartoloBF NA 1983 7}]\label{FSMeq lem3.5}
 Let $X$ be an infinite dimensional Banach space, $X = Y \oplus Z,$
where $ Y$ is finite dimensional. If $J\in  C^1(X,\mathbb{ R}) $ satisfies $(C)_c$-condition for all
$c > 0, $ and\\[2mm]
$(J1)$ $J(0) = 0, J(-u) = J(u)$ for all $u \in  X;$\\[2mm]
$(J2)$ there exist constants  $\rho,\delta > 0$
such that $J|_{\partial
B_\rho\cap Z_m}\geq \delta >0;$\\[2mm]
$(J3)$ for any finite dimensional
subspace $\widetilde{E}\subset E,$ there is $ R = R(\widetilde{E} ) > 0$ such that
$ J(u)\leq 0, \ \  \ \forall  u\in \widetilde{E}\backslash B_R;$ \\[2mm]
then $J$ possesses an unbounded sequence of critical values.
\end{lem}
\begin{proof}[\bf Proof of Theorem \ref{FSMeq thm1.1}.]  Let $X = E, Y = Y_m $ and $Z = Z_m. $ By Lemmas \ref{FSMeq lem3.1}, \ref{FSMeq lem3.4} and Corollary \ref{FSMeq cor3.1}, all conditions of
Lemma \ref{FSMeq lem3.5} are satisfied. Thus, problem \eqref{FSMeq 1.1} and  \eqref{FSMeq 1.2} possesses infinitely many nontrivial solutions.
\end{proof}
% 参考文献


\begin{thebibliography}{99}

\bibitem{DM PRSE 2004 134} T. D'Aprile, D. Mugnai, Solitary waves for nonlinear
Klein-Gordon-Maxwell and Schr\"odinger-Maxwell equations, Proc. Roy. Soc. Edinburgh Sect. A {\bf
134} (2004), 893--906.

\bibitem{BF TMNA 1998 11} V. Benci, D. Fortunato, An eigenvalue problem for the Schr\"odinger-Maxwell
equations. Topol. Methods Nonl. Anal. {\bf 11} (1998) 283--293.

\bibitem{AR CCM 2008 10}  A. Ambrosetti, D. Ruiz, Multiple bound states for the Schr\"{o}dinger-Poisson problem. Commun. Contemp. Math. 10 (2008) 391-404.
\bibitem{GMC CAA 2003 7} G.M. Coclite, A multiplicity result for the nonlinear Schr\"{o}dingerMaxwell equations. Commun. Appl. Anal. 7 (2003) 417-423.
\bibitem{TD ANS 2002} T. D'Aprile, Non-radially symmetric solution of the nonlinear Schr\"{o}dinger equation coupled with Maxwell equations. Adv. Nonlinear Stud. 2 (2002)
177-192.
\bibitem{HK NA 2007 27}H. Kikuchi, On the existence of solution for elliptic system related to the Maxwell-Schr\"{o}dinger equations. Nonlinear Anal. 27 (2007) 1445-1456.
\bibitem{DR JFA 2006 237}D. Ruiz, The Schr\"{o}dinger-Possion equation under the effect of a nonlinear local term. J. Funct. Anal. 237 (2006) 655-674.

\bibitem{LJJ PRCSE 1999 129} L. Jeanjean, On the existence of bounded Palais-Smale sequences and application to a Landesman-Lazer type problem set on $\mathbb{R}^N .$ Proc. Roy. Soc.
Edinburgh Sect. A 129 (1999) 787-809.
\bibitem{LZ FZ NA 2009 70} L. Zhao, F. Zhao, Positive solutions for Schr\"{o}dinger-Poisson equations with a critical exponent, Nonlinear Anal. 70 (2009) 2150-2164.




\bibitem{AP jmaa 2008 345} A. Azzollini, A. Pomponio, Ground state solutions for the nonlinear Schr\"{o}dinge-Maxwell equations. J. Math. Anal. Appl. 345 (2008) 90-108.

\bibitem{CS TCL NA 2009 71}  Chen, S .J.,  Tang, C.-L.: High energy solutions for the superlinear Schr\"{o}dinger-
Maxwell equations. Nonlinear Anal. 71(2009) 4927-4934 .

\bibitem{LiQD Weizl NA 2010 72} Li, Q., Su, H., Wei, Z.: Existence of infinitely many large solutions for the nonlinear Schr\"{o}dinger-Maxwell equations. Nonlinear Anal. 72 (2010) 4264-4270.

   \bibitem{SunJ jmaa 2012 390}Sun, J: Infinitely many solutions for a class of sublinear Schr\"{o}dinger-Maxwell
equations. J. Math. Anal. Appl. 390 (2012) 514-522.
\bibitem{HuangWN TangXH RM 2014 65}
Wen-nian Huang, X.H. Tang,  The existence of infinitely many  solutions for the nonlinear Schr\"{o}dinger-Maxwell equations.
Results. Math. 65( 2014) 223-234.


\bibitem{WillemM 1996}  Willem, M.: Minimax Theorems. Birkh¨auser, Boston (1996).

\bibitem{ZouWM MM 2001 104} Zou, W.: Variant fountain theorems and their applications. Manuscripta
Math. 104 (2001) 343-358.

\bibitem{ChangXJ JMP2013}X. Chang,  Ground state solutions of asymptotically linear fractional Schr\"{o}inger equations. J Math Phys. 54 (2013) 061504.
\bibitem{PFelmer AQuaas and JGTan PRSE 2012 142} P. Felmer, A. Quaas, and J. G. Tan, Positive solutions of nonlinear Schr\"{o}dinger equation with the fractional Laplacian.
Proc. - R. Soc. Edinburgh, Sect. A: Math. 142 (2012) 1237-1262 .
\bibitem{Zifei Shen and Fashun Gao AAA 2014} Zifei Shen and Fashun Gao, On the Existence of Solutions for the Critical Fractional
Laplacian Equation in $\mathbb{R}^N.$ Abstract and Applied Analysis,  2014, Article ID 143741, 10 pages.



\bibitem{HajaiejH YuX ZhaiZ jmaa 2012}Hajaiej H, Yu X, Zhai Z. Fractional Gagliardo-Nirenberg and Hardy inequalities under Lorentz norms. J. Math. Anal. Appl.  396 (2012) 569-577.
\bibitem{Elliott H Lieb  Michael Loss Analysis 2001}  Elliott H. Lieb, Michael Loss,  Analysis, Second edition (Graduate Studies in Mathematics 14)-AMS Bookstore (2001).

\bibitem{Yosida K 1999} Yosida, K.: Functional Analysis, 6th edn. Springer-Verlag, New York (1999).

\bibitem{BartoloBF NA 1983 7} Bartolo, T., Benci, V., Fortunato, D.: Abstract critical point theorems and applications to some nonlinear problems with strong resonance at infinity. Nonlinear
Anal. 7, 241-273 (1983).

\end{thebibliography}
\end{document}